\documentclass[a4paper,12pt]{amsart}

\usepackage[utf8]{inputenc}
\usepackage[T1]{fontenc}
\usepackage{amsmath,amssymb,amsfonts,amsthm,mathrsfs}
\usepackage{geometry}
\usepackage{graphicx}
\usepackage{float}
\usepackage[english]{babel}
\usepackage{tikz}
\usetikzlibrary{arrows}
\usetikzlibrary{decorations.markings}

\theoremstyle{plain}
\newtheorem{df}{Definition}
\newtheorem{thm}{Theorem}
\newtheorem{lem}{Lemma}
\newtheorem{prop}{Proposition}
\newtheorem{ex}{Example}

\newtheorem{cor}{Corollary}
\newtheorem*{thm*}{Theorem}
\newcommand{\mot}{}
\newtheorem*{thmref_interne}{\mot{}}

\newcommand\mb{\mathbb}
\newcommand\mc{\mathcal}
\newcommand\mf{\mathfrak}
\newcommand\mr{\mathrm}
\newcommand\ms{\mathscr}

\newcommand\F{\mc{F}}
\newcommand\G{\mc{G}}

\renewcommand\emph[1]{\textup{\textbf{#1}}}

\title{
About $\mc{C}^\infty$ foliations by holomorphic curves on complex surfaces
}

\author{Olivier Thom}
\date{\today}

\begin{document}

\begin{abstract}
We study those real $\mc{C}^\infty$ foliations in complex surfaces whose leaves are holomorphic curves.
The main motivation is to try and understand these foliations in neighborhoods of curves: can we expect the space of foliations in a fixed neighborhood to be infinite-dimensional, or are there some contexts under which every such foliation is holomorphic ?

We give some restrictions and study in more details the geometry of foliations whose leaves belong to a holomorphic family of holomorphic curves.
In particular, we classify all real-analytic foliations on neighborhoods of curves which are locally diffeomorphic to foliations by lines, under some non-degeneracy hypothesis.
\end{abstract}

\maketitle

\section{Introduction}

In this paper, we study real $\mc{C}^\infty$ foliations in complex surfaces whose leaves are holomorphic curves; we will call such foliations \emph{semiholomorphic} for short in this text.

These foliations appear naturally in many different contexts, were studied from different points of views, sometimes under different names; after giving some definitions, we will briefly review some problems related to semiholomorphic foliations and recall some results which might be of interest.

\subsection{Equations}

Let $U\subset \mathbb{C}^2$ be an open set and $\F$ a $C^\infty$ real codimension 2 foliation on $U$.
The foliation $\F$ is called \emph{semiholomorphic} if the subsheaf $T\F \subset TU$ consist of holomorphic directions: $T\F \subset T^{1,0}U$.

Suppose that $U$ is equipped with holomorphic coordinates $(x,y)$ and that the foliation is smooth and nowhere vertical: $ \frac{\partial}{\partial y}\notin T_p\F$ for every $p\in U$.
Then the foliation can be described by the (1,0)-form $\omega = dy - \lambda dx$, where $\lambda\in C^\infty(U, \mathbb{C})$ is the slope.
This (1,0)-form satisfies the integrability condition
\begin{equation}
\label{eq_intgb_omega}
\omega \wedge \overline{\omega} \wedge d \omega = 0.
\end{equation}
In terms of the function $\lambda$, this equations writes:
\begin{equation}
\label{eq_intgb_lambda}
	\frac{\partial_{}\lambda}{\partial_{}\bar{x}} + \overline{\lambda} \frac{\partial_{}\lambda}{\partial_{}\bar{y}} = 0.
\end{equation}

Conversely, a field of holomorphic directions written as the kernel of a (1,0)-form $\omega$ defines a semiholomorphic foliation if and only if the integrability condition \eqref{eq_intgb_omega} is satisfied.

Suppose that $\F$ is smooth on $U$ and that $L$ is a real codimension 2 subvariety of $U$ invariant by $\F$.
Then at every point $p\in L$, the tangent space $TL$ is a complex direction.
This shows that $L$ is a complex curve, and that a semiholomorphic foliation is exactly a $C^\infty$ foliation by holomorphic curves.

\begin{ex}
\label{ex_1}
Consider the (1,0)-form $\mr{Im}(x)dy - \mr{Im}(y)dx$. It satisfies equation \eqref{eq_intgb_omega} so defines a semiholomorphic foliation.
This foliation has singular set $\mathbb{R} \times \mathbb{R}$. 
The leaf passing through $(x_0,y_0)\notin \mathbb{R}^2$ has equation 
\[
\begin{aligned}
y &= y_0 + \frac{\mr{Im}(y_0)}{\mr{Im}(x_0)}(x-x_0)\\
 &= \frac{\mr{Im}(y_0)}{\mr{Im}(x_0)}x + \left( \mr{Re}(y_0) - \frac{\mr{Im}(y_0)}{\mr{Im}(x_0)}\mr{Re}(x_0) \right).
\end{aligned}
\]
Remark that the leaves of this foliation are exactly the complex affine lines with real parameters.
\end{ex}

\subsection{Holomorphic motion}
\label{sec_motion}

Consider an open set $U\subset \mathbb{C}^2$ and a semiholomorphic foliation $\F$ in $U$.
Fix a transverse holomorphic fibration $\{x=cte\}$ on $U$ and a local trivialization $(x,y)$ of this fibration.
Write $T_x$ the fiber above $x$.
Consider also an origin $0\in U$ and suppose that the curve $\{y=0\}$ is a leaf of $\F$.

For each $x$, consider the holonomy transport $\varphi_x: T_0 \rightarrow T_x$ obtained by following the leaves of $\F$.
This is a family of $\mc{C}^\infty$ diffeomorphisms depending holomorphically on the parameter $x$.
Using the trivialization, we can consider $\varphi_x$ as a family of germs of diffeomorphisms in the variable $y$.
Conversely, given any family $\varphi_x$ of germs of diffeomorphisms depending holomorphically on the parameter $x$, we can construct a semiholomorphic foliation transverse to the fibration $\{x=cte\}$ by taking trajectories of points $y_0\in T_0$.

To find the equivalence with the previous point of view, consider a point $(x,y)\in U$ and write $y_0$ the point of intersection between $T_0$ and the leaf passing through $(x,y)$.
This means that $y=\varphi_x(y_0)$, or rather $y_0 = \varphi_x^{-1}(y)$.
Then the slope of the foliation is
\[
\lambda(x,y) = \frac{\partial_{}\varphi}{\partial_{}x}(x,\varphi_x^{-1}(y)),
\]

\begin{ex}
In Example \ref{ex_1}, the holonomy transport between transversals above $x_0=i$ and $x$ writes 
\[
\varphi_x(y) = \mr{Re}(y) + x\mr{Im}(y).
\]
\end{ex}

This point of view appears naturally in holomorphic dynamics, see for example \cite{mss} and \cite{sullivan+thurston}.
The main problem studied from this point of view is the extension of the semiholomorphic foliation in the transverse direction; the results show that the transversal behaviour is very similar to that of $\mc{C}^\infty$ objects (for example, the proof of \cite[\S 2]{sullivan+thurston} involves partitions of unity).
The most complete theorem obtained in this direction seems to be \cite{slodkowski_hulls}: in $\mb{D}\times \mathbb{C}$, any set of disjoint complex curves $(C_e)_{e\in E}$ transverse to the fibers $\mathbb{C}$ can be extended to a semiholomorphic foliation in $\mb{D}\times \mathbb{C}$ transverse to the vertical fibration.

\subsection{Levi-flat hypersurfaces and Ueda theory}
\label{sec_intro_ueda}

An interesting motivation for semiholomorphic foliations comes from Ueda theory, and more generally the study of Levi-flat hypersurfaces.
Recall that a smooth $\mc{C}^\infty$ real hypersurface $H$ in a complex surface $S$ is called Levi-flat if the field of complex directions $TH\cap J(TH)$ is integrable on H, where $J:TS \rightarrow TS$ denotes the operator given by multiplication by $i$.

It follows that $H$ is foliated by complex curves.
In the $\mc{C}^\infty$ category, we cannot say anything a priori about the transverse regularity of this foliation; however, when a Levi-flat hypersurface $H$ is real-analytic, then by \cite{cartan_geometrie_pseudo_conforme_1} this foliation can be extended to a holomorphic foliation in a neighborhood of $H$.
Note that, if we have a smooth $\mc{C}^\infty$ foliation by smooth Levi-flat hypersurfaces $\mc{H}$ in $S$, the field of complex directions $T\mc{H}\cap J(T\mc{H})$ defines a smooth semiholomorphic foliation.

This kind of objects appear naturally in the study of neighborhoods of compact complex curves: by a theorem of T. Ueda \cite{ueda}, for a smooth compact curve $C$ in a complex surface $S$ with torsion normal bundle, either $C$ admits a system of strictly pseudo-concave neighborhoods, either $C$ admits a logarithmic 1-form $\omega$ with poles along $C$ and purely imaginary periods.
The foliation defined by this 1-form is a smooth holomorphic foliation admitting $C$ as a leaf, and with unitary holonomy.
It follows that there exists in $S$ a foliation by Levi-flat hypersurfaces, each of which is the border of a neighborhood of $C$, and that the foliation defined by $\omega$ is the semiholomorphic foliation tangent to it.
Thus, when $N_C$ is torsion, any semiholomorphic foliation defined in a neighborhood of $C$, tangent to a foliation by Levi-flat hypersurfaces of this kind, is necessarily holomorphic.

This theorem is also true for smooth compact curves $C$ with generic normal bundle of degree 0, but in the non-generic cases the question is still open (typically, the non-generic cases correspond to neighborhood $S$ of $C$ which are formally but not analytically linearizable).

An interesting question would then be to understand when a neighborhood $S$ of a compact curve $C$ admits smooth semiholomorphic foliations tangent to $C$ which are not holomorphic, and in particular if these semiholomorphic foliations are exceptional objects or if any neighborhood $S$ admits one of them.

\subsection{Teichmüller theory}

Holomorphic motions, and thus semiholomorphic foliations, appear naturally in the context of Teichmüller theory, giving rise to some intersting examples, see for example \cite{bers+royden} and \cite{mcmullen_hilbert} for more details.

Let us just give one example to explain the link between the two.
Consider the plane $\mathbb{C}$ of the variable $y$, and a polygon $P_i$ inside it equipped with identifications of opposite parallel sides, giving rise to a translation surface $C_i$.
Suppose that one side is included in the real axis and that $P_i$ is contained in the upper half-plane $\mb{H}_y$.
Now, for every $x\in \mb{H}$, consider the linear application $\varphi_x\in \mr{GL}_2(\mathbb{R})$ fixing the real axis and sending $i$ to $x$.
The application $\varphi_x$ sends the polygon $P_i$ to some polygon $P_x$ defining a translation surface $C_x$.

\begin{figure}[H]
\includegraphics[scale=0.3]{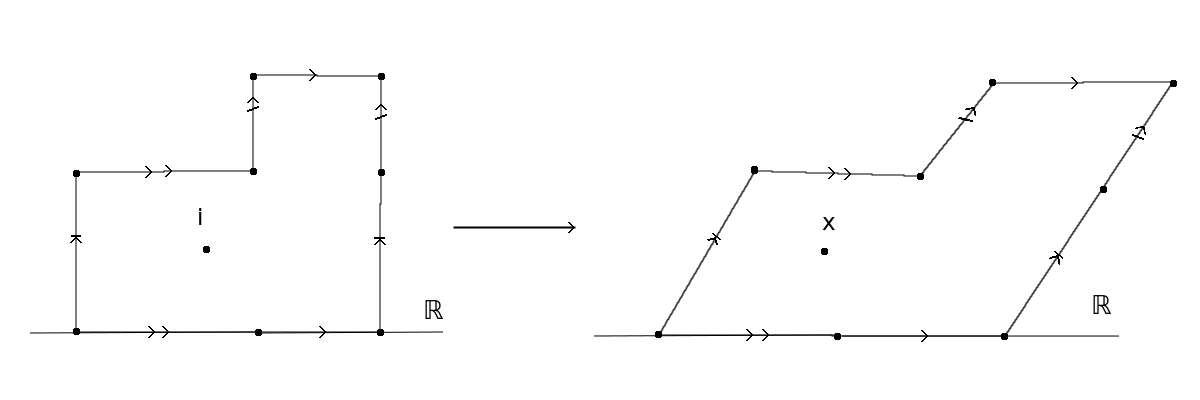}
\caption{A deformation of translation surfaces}
\end{figure}

We can consider the space of pairs $(x,y)\in \mb{H}\times \mathbb{C}$ as a complex surface equipped with a holomorphic motion $(\varphi_x)$; the set of points $(x,y)$ with $y\in P_x$ is stable by the motion and we can see the union of the translation surfaces $S=\cup_x C_x$ as a quotient of it.
The semiholomorphic foliation $\F$ defined by $(\varphi_x)$ descends to a semiholomorphic foliation on the bundle of translation surfaces $S$.
Note that the application $\varphi_x$ writes $\varphi_x(y) = \mr{Re}(y) + x\mr{Im}(y)$, so that locally the foliation $\F$ is in fact the one of Example \ref{ex_1}.

When the translation surfaces admit a lattice of symmetries $L$, the foliation descends to a foliation on the surface $S/L$.
However by standard arguments, the lattice $L$ cannot be cocompact. 
Indeed if $S/L$ were compact, the length of the shortest loop in $C_x$, being a continuous function of $x$, should be bounded by below.
But this length tends to zero for $C_{(ti)}$ when $\mb{R}\ni t\to 0$.

\subsection{Monge-Ampère foliations}

A large class of example is also given by Monge-Ampère foliations, in the sense of \cite{bedford+kalka}: given a real plurisubharmonic function $f\in \mc{C}^\infty(U)$ in some open set $U\subset \mathbb{C}^2$, introduce the complex hessian $\Omega = \partial_{}\overline{\partial_{}}f$: it is a $(1,1)$-form and whenever $\Omega \wedge \Omega = 0$, the field of complex directions $X$ defined by the equation $i_X \Omega = 0$ is a semiholomorphic foliation.
In dimension 2, each foliation can be obtained this way, but for foliations of higher codimension, the two notion are no more equivalent (see \cite{bedford+kalka} and \cite{duchamp+kalka} for more details).

In the article \cite{duchamp+kalka}, the authors also studied general semiholomorphic foliations and obtained some interesting results.
More precisely, following the ideas of \cite{bedford+burns} for Monge-Ampère foliations, they studied Bott's partial connection of a semiholomorphic foliation $\F$ which is not holomorphic, and showed that it induces a connection of negative curvature on the normal bundle $N^{1,0}\F$.
This allows to define an intrinsic metric of curvature -4 on the leaves of $\F$.

These two facts have interesting consequences, for example that a semiholomorphic foliation whose leaves are compact is always holomorphic \cite[Example 6.6]{duchamp+kalka} (see also \cite{winkelmann_semiholo} for another proof of this fact using different tools), or that a semiholomorphic foliation whose leaves are parabolic is necessarily holomorphic \cite{kalka+patrizio}.

\subsection{Summary of this article}

We begin in section \ref{sec_bott} by recalling how the antiholomorphic part $\eta_\F$ of Bott's partial connection gives a hyperbolic singular metric $|\eta_\F|^2$ on the leaves of $\F$, as was already noted in \cite{duchamp+kalka}.
We then carry on studying the metric $|\eta_\F|^2$ to obtain finer results.

It follows from \cite{duchamp+kalka} that if a smooth semiholomorphic foliation $\F$ on a surface $S$ has a compact curve $C$, then $C\cdot C<0$.
In Theorem \ref{thm_1-g}, we show also that $C\cdot C\geq 1-g$, with equality when the form $i^*_C\eta_\F$ has no zeroes (remember that the metric $|\eta_\F|^2$ is well-defined, so that even if $\eta_\F$ is not a well-defined 1-form on $C$, its zeroes are well-defined).

Also, we note in Theorem \ref{thm_complete_leaves} that under some reasonable hypotheses, the leaves of a non-holomorpic semiholomorphic foliation should be complete for the metric $|\eta_\F|^2$.
This reminds the article \cite{brunella_leafwise_metric} with the differences that in the semiholomorpic context, the metric arises naturally from the foliation, but can be a singular metric.

In section \ref{sec_semirank} we introduce some other geometric invariants of semiholomorphic foliations $\F$.
Informally, we consider the smallest holomorphic family of holomorphic curves containing the leaves of $\F$, call it \emph{system of curves} defined by $\F$ and call \emph{semirank} its number of parameters.
These notions are invariant under biholomorphisms, and although a generic $\mc{C}^\infty$ foliation will be of infinite semirank, we show in Proposition \ref{prop_semirank_analytic} that real-analytic semiholomorphic foliations are of semirank 2.
Note also that all the examples where the leaves are lines are of semirank 2, so this motivates us to restrict our attention to foliations of semirank 2 for the rest of the article.

This additional hypothesis gives us some new tools to study these objects, for example we can use the duality between 2-parameter families of curves.
This topic has a long history and we will refer the reader to \cite{fl_projective_structures} and the references therein for more details.
In a local context, note that if we fix a 2-parameter family of curves, the set of parameters is a complex surface $\check{U}$ and that a semiholomorphic foliation gives a real surface $S_\F$ in $\check{U}$.
Studying $\F$ then reduces to studying $S_\F$.
In particular if $\F$ is the universal cover of a semiholomorphic foliation in the neighborhood of a curve $C$, then the fundamental group of $C$ acts by holomorphic automorphisms of the system of curves, and its action on the dual $\check{U}$ stabilizes the real surface $S_\F$.
These conditions are very strong and we will use them extensively.

As we saw in Theorem \ref{thm_complete_leaves}, we can expect interesting examples to have complete leaves.
We try at the end of this section to describe what a semiholomorphic foliation in an open set $U\subset \mathbb{C}^2$ with complete leaves should look like in semirank 2, which should be the generic behaviour for a semiholomorphic foliation of any semirank.

In the last section, we consider foliations whose system of curves is a projective structure in the sense of \cite{fl_projective_structures}.
These projective structures are better understood than general families of curves, and as we saw they can arise naturally from global contexts.
In particular, we classify all real-analytic foliations $\F$ on neighborhoods of compact curves $C$ which are locally diffeomorphic to foliations by lines, under the hypothesis that $i_C^*\eta_\F$ is not identically zero, see Theorems \ref{thm_foliation_lines_examples} and \ref{thm_generic_foliation_lines}.

\subsection{Acknowlegments}

The author wishes to thank IMPA and PNPD-CAPES for their support.
Thanks also to J. V. Pereira for useful advices and to B. Deroin for pointing out the importance of the subject, and its links with Teichmüller theory.

\section{Bott's partial connection}
\subsection{Local expression}
\label{sec_bott}

Suppose in this section that $\F$ is transverse to the fibration $x=cte$.
Consider the form $\omega_0 = dy - \lambda(x,y) dx$ defining the semiholomorphic foliation $\F$.
Note that 
\[
\begin{aligned}
d \omega_0 &= \frac{\partial_{}\lambda}{\partial_{}y} dx \wedge dy + \frac{\partial_{}\lambda}{\partial_{}\bar{x}}dx \wedge d\bar{x} + \frac{\partial_{}\lambda}{\partial_{}\bar{y}}dx \wedge d\bar{y}\\
&= \left( \frac{\partial_{}\lambda}{\partial_{}y} dx \right) \wedge \omega_0 + \left(\frac{\partial_{}\lambda}{\partial_{}\bar{y}} dx \right) \wedge \overline{\omega_0},
\end{aligned}
\]
and that for any function $f$, if $\omega = f \omega_0$, then
\[
d \omega = \left( \frac{df}{f} + \frac{\partial_{}\lambda}{\partial_{}y}dx \right) \wedge \omega + \left(\frac{f^2}{|f|^2}  \frac{\partial_{}\lambda}{\partial_{}\bar{y}} dx \right) \wedge \overline{\omega}.
\]

If not for the factor $\frac{f^2}{|f|^2}$, the 1-form $\eta := \frac{\partial_{}\lambda}{\partial_{}\bar{y}}dx$ would be well-defined modulo some multiples of $\omega$ and $\overline{\omega}$, and would define a $(1,0)$-form on the leaves of $\F$.
Since this factor is of modulus 1, the metric on the leaves $|\eta|^2 = i\eta \wedge \overline{\eta}$ depends only on the foliation $\F$.
We will write this 1-form $\eta_\F$ when the defining form $\omega$ is clear or irrelevant and call it the anti-holomorphic part of Bott's connection.

The two following lemmas are immediate.

\begin{lem}
\label{lem_bott_1}
The foliation $\F$ is holomorphic if and only if $\eta_\F = 0$; it is holomorphic at order 1 along a leaf $L$ if and only if $i_L^*\eta_\F = 0$.
\end{lem}

\begin{lem}
\label{lem_bott_2}
If $\Phi$ is a holomorphic germ of diffeomorphism around the origin, if $\eta$ and $\widetilde{\eta}$ are the anti-holomorphic parts of Bott's connection applied to $\omega$ and $\Phi^* \omega$ respectively, then $\widetilde{\eta} = \Phi^*\eta$.
\end{lem}

%
%
%

\subsection{Tangential behaviour}

Let us study more explicitely the coefficient $\frac{\partial_{}\lambda}{\partial_{}\overline{y}}$ in the form $\eta_\F$.
First, let us introduce the differential operator
\[
\partial_{\F} := \frac{\partial_{}}{\partial_{}x} + \lambda \frac{\partial_{}}{\partial_{}y}.
\]
We check that for every function $f: \mathbb{C}^2 \rightarrow \mathbb{C}$, the integrability of $\F$ guarantees $\partial_{\F} \overline{\partial_{\F}} f = \overline{\partial_{\F}}\partial_{\F}f$ so we can introduce the real operator 
\[
\Delta_\F =\partial_{\F}\overline{\partial_{\F}}.
\]
Explicitely, we have $\Delta_\F f = \Delta_x f + |\lambda|^2 \Delta_y f + \lambda \frac{\partial_{}^2 f}{\partial_{}\bar{x} \partial_{}y} + \bar{\lambda} \frac{\partial_{}^2 f}{\partial_{}x \partial_{}\bar{y}}$.
This operator corresponds to the Laplacian in restriction to the leaves of $\F$: indeed if $\{y=0\}$ is a leaf of $\F$, then we have $\Delta_\F = \Delta_x$ on this leaf.

Put 
\[
a := \frac{\partial_{}\lambda}{\partial_{}y},\quad b:= \frac{\partial_{}\lambda}{\partial_{}\overline{y}}.
\]
The operator $\overline{\partial_{\F}}$ does not commute with $\frac{\partial_{}}{\partial_{}y}$ (nor with $\frac{\partial_{}}{\partial_{}\bar{y}}$), more precisely we have
\[
\left\{ \begin{aligned}
\overline{\partial_{\F}}a + \bar{b}b &= 0\\
\overline{\partial_{\F}}b + \bar{a}b &= 0.
\end{aligned} \right.
\]

Introduce $\beta = \mr{log}(b) = \beta_1 + i \beta_2$ so that the second equation writes 
\[
\bar{a} = - \overline{\partial_{\F}} \beta,
\]
and the first equation gives
\begin{equation}
\label{eq_bott_harmonicity}
\left\{ \begin{aligned}
\Delta_\F \beta_1 &= \mr{exp}(2\beta_1)\\
\Delta_\F \beta_2 &= 0.
\end{aligned} \right.
\end{equation}

As we see in \cite{duchamp+kalka}, this equation gives:

\begin{lem}
\label{lem_bott_3}
The metric $|\eta_\F|^2$ on the leaves is of curvature -4 whenever it is not zero.
\end{lem}

Equation \eqref{eq_bott_harmonicity} will imply that $b$ is very regular along the leaves; in fact we can prove this directly using the point of view of holomorphic motions.
Denote as before $\varphi_x$ the holonomy transport of $\F$ between fibers of $\{x=cte\}$; the antiholomorphic part of Bott's connection is given by
\[
\frac{\partial_{}\lambda}{\partial_{}\bar{y}}(x,y) = \frac{\partial_{}^2 \varphi}{\partial_{}x \partial_{}y} \left(x, \varphi_x^{-1}(y) \right)\frac{\partial_{}\varphi_x^{-1}}{\partial_{}\bar{y}}(y) + \frac{\partial_{}^2 \varphi}{\partial_{}x \partial_{}\bar{y}} \left(x,\varphi_x^{-1}(y) \right) \overline{\frac{\partial_{}\varphi_x^{-1}}{\partial_{}y}(y)}.
\]

In restriction to a leaf $\{y=0\}$, we can approximate $\varphi_x$ by its linear part: $\varphi_x(y) = l_x(y) + O(|y|^2) = u(x)y + v(x)\bar{y} + O(|y|^2)$, where $u$ and $v$ are holomorphic functions of the parameter $x$.
Note that 
\[
l_x^{-1}(y) = \frac{\overline{u}}{|u|^2-|v|^2}y - \frac{v}{|u|^2-|v|^2}\bar{y}.
\]
Thus we get in first order approximation
\[
\lambda(x,y) = u'(x) \frac{\overline{u}y-v\bar{y}}{|u|^2-|v|^2} + v'(x) \frac{u\bar{y}-\bar{v}y}{|u|^2-|v|^2} + O(|y|^2),
\]
and
\[
\frac{\partial_{}\lambda}{\partial_{}\bar{y}}(x,0) = \frac{uv'-u'v}{|u|^2-|v|^2}(x).
\]

In particular, we see that the border of the leaf $\{y=0\}$ is given by the real analytic curve $\{|u|^2=|v|^2\}$; we can also deduce the following:

\begin{lem}
\label{lem_bott_4}
On a leaf of $\F$, the set of points $x$ where $b(x)=0$ is either discrete or the whole leaf.
Moreover, if $x$ is an isolated zero of $b$, then $\frac{\partial_{}b}{\partial_{}\bar{x}}(x) = 0$.
\end{lem}

\subsection{Global setting}

In this section, $S$ is a complex surface and $\F$ a smooth semiholomorphic foliation on $S$; we fix a covering $S = \cup U_i$, local coordinates $(x_i,y_i)$ with $\F$ transverse to the fibration $x_i=cte$, and a (1,0)-form $\omega_i = f_i \cdot (dy_i - \lambda_i dx_i)$ defining $\F$ on each $U_i$.

The form $\eta_\F$ depends on the choice of the representant $\omega_i$ of $\F$, but as we have seen in section \ref{sec_bott}, it is a section of a line bundle $T^{1,0}\F\otimes L$ where the transition functions of $L$ are $\mc{C}^\infty$ functions of modulus 1.
Putting together lemmas \ref{lem_bott_1}, \ref{lem_bott_2}, \ref{lem_bott_3} and \ref{lem_bott_4}, we obtain the following:

\begin{prop}
The metric $|\eta_\F|^2$ is intrinsically defined as a metric on the leaves of $\F$.
On each leaf, this metric is either identically zero, or only has isolated zeroes.
If it is not identically zero, then it has curvature -4 away from its zeroes.
\end{prop}

The following corollary, and Theorem \ref{thm_levi_flat} were already consequences of \cite{duchamp+kalka} and \cite{kalka+patrizio}.


\begin{cor}
Let $S$ be a complex surface with a smooth holomorphic fibration $\pi$ towards an elliptic curve.
Then every smooth semiholomorphic foliation transverse to $\pi$ is holomorphic.
\end{cor}


\begin{thm}
\label{thm_levi_flat}
Suppose $S$ is a neighborhood of an elliptic curve $C$ and there exists a singular $\mc{C}^\infty$ foliation $\mc{H}$ by compact Levi-flat hypersurfaces such that $C$ is invariant by $\mc{H}$, the foliation $\mc{H}$ is smooth outside $C$ and every leaf of $\mc{H}$ is the border of a neighborhood of $C$.
Suppose moreover that for every sequence of points $p_n \in S$ with $p_n \rightarrow p_\infty\in C$ and such that the sequence $T_{p_n}\mc{H}$ has a limit $H_\infty$, then $T_{p_\infty}C\subset H_\infty$.

Then $\mc{H}$ is tangent to a holomorphic foliation. 
\end{thm}

\begin{proof}
The field of complex directions $\F = T^{1,0}\mc{H}$ is integrable so $\F$ is a semiholomorphic foliation.
Under the hypotheses, for every local fibration $\{x_i=cte\}$ transverse to $C$ on an open set $U_i\subset S$, this fibration is also transverse to $\mc{H}$ in a neighborhood of $C\cap U_i$.
Thus it is transverse to the foliation $\F$, which implies in particular that $\F$ is smooth in a neighborhood of $C$.

When $|\eta_\F|^2$ is not identically zero, it is a metric of curvature $-4$, so we only need to prove that every leaf of $\F$ is uniformized by $\mathbb{C}$.
Choose a Kähler form $\omega$ on $S$ so that we can compute the curvature of the leaves using $\omega$.
Consider a smooth $\mc{C}^\infty$ foliation $\mc{G}$ transverse to $\F$.
For any base point $p_0\in C$, if we denote by $G_0$ the leaf of $\G$ passing through $p_0$, we get an application $f:G_0\times \widetilde{C} \rightarrow \widetilde{S}$ following the leaves of $\F$, where $\widetilde{C}$ is the universal cover of $C$ and $\widetilde{S}$ that of $S$.
Consider the map $f(y_0,\cdot): C\cap U_i \rightarrow L_0\cap U_i$ on an small open set $U_i$, where $y_0\in G_0$ and $L_0$ is the leaf of $\F$ passing through $y_0$: by continuity of the metric induced by $\omega$ on the leaves of $\F$, this application is a quasi-isometry, and we can choose constants of quasi-isometry which do not depend on $y_0$.
Since there are a finite number of these $U_i$, we see that $f(y_0,\cdot)$ is a quasi-isometry between $\widetilde{C}$ and the universal cover $\widetilde{L}$ of the leaf $L$ of $\F$ passing through $y_0$, for any $y_0\in G_0$.
We conclude that any leaf is uniformized by $\mathbb{C}$ as needed.
\end{proof}

Remark that the proof is somewhat more general: we only need a smooth semiholomorphic foliation whose leaves stay inside a neighborhood of $C$.

We will see in the explicit examples of section \ref{sec_ex_lines} that for a semiholomorphic foliation with a compact leaf, the normal and tangent bundles of this leaf are closely related; in general we can prove the following.

\begin{thm}
\label{thm_1-g}
Suppose $C$ is a compact leaf of genus $g$ of a smooth semiholomorphic foliation $\F$ on a surface $S$, and $i_C^*\eta_\F\not\equiv 0$.
Write $n$ the number of zeroes of the form $i_C^*\eta_\F$, counted with multpilicity.
Then $C\cdot C = 1-g + \frac{n}{2}$.
\end{thm}

Note in particular that Camacho-Sad's theorem is false for smooth semiholomorphic foliations.

\begin{proof}
The approximation at first order along $C$ of $\F$ is a smooth semiholomorphic foliation on the normal bundle of $C$ in $S$.
The proposition only depends on this first-order approximation, so we can suppose that $S$ is the normal bundle of $C$.

Consider local charts $(x_i,y_i)$ with 
\[
(x_j,y_j) = \Phi_{ji}(x_i,y_i) = (\alpha_{ji}(x_i), \beta_{ji}(x_i)y_i).
\]
Consider local $(1,0)$-forms $\omega_i = dy_i - \lambda_i dx_i$ defining $\F$, and the corresponding antiholomorphic parts of Bott's connection $\eta_i = \frac{\partial_{}\lambda_i}{\partial_{}\overline{y_i}}dx_i$.
From Lemma \ref{lem_bott_2} we know that
\[
i_C^*(\Phi_{ji}^*\eta_j) = \frac{\beta_{ji}^2}{|\beta_{ji}|^2}i_C^*\eta_i.
\]
Thus if we write $b_i = \frac{\partial_{}\lambda_i}{\partial_{}\overline{y_i}}\vert_C$, we get
\[
b_j\circ\alpha_{ji} = \frac{\beta_{ji}^2}{|\beta_{ji}|^2 \alpha'_{ji}} b_i.
\]
The cocycle $|\beta_{ji}|^{-2}$ is real positive so its sections do not vanish; the cocycle $\beta_{ji}$ defines the normal bundle and $\alpha'_{ji}$ the tangent bundle.
On the other hand we see by Lemma \ref{lem_bott_4} that $b_i$ is a $\mc{C}^\infty$ function with holomorphic zeroes.
It follows that $\mr{deg}(\Omega^1_C\otimes N^{2})=n$ where $n$ is the number of zeroes of the section $b$, hence the result.
\end{proof}

\begin{ex}
Consider a germ of surface $(S,C)$ along a compact curve $C$ of genus $g\geq 2$, and suppose there exists on $S$ a smooth semiholomorphic foliation $\F$ leaving $C$ invariant with $i_C^*\eta_\F$ nowhere vanishing.
From the theorem we know that $C\cdot C = 1-g$.

Consider then two curves $c_1,c_2$ in $S$ cutting $C$ transversally at two points $p_1\neq p_2$, and the ramified covering $(\widetilde{S},\widetilde{C})$ of order 2 over $(S,C)$ ramifying along $c_1$ and $c_2$.
The genus of $\widetilde{C}$ is $\widetilde{g} = 2g$; its self-intersection is $\widetilde{C}\cdot \widetilde{C} = 2-2g = 2 - \widetilde{g}$, so that $\mr{deg}(\Omega^1_C\otimes N_C^2) = (2 \widetilde{g}-2) + 2 (2 - \widetilde{g}) = 2$.
We see from Lemma \ref{lem_bott_2} that $i_{\widetilde{C}}^*\widetilde{\eta}$ does not vanish at points different from $p_1,p_2$, and a similar computation shows that it admits simple zeroes at $p_1$ and $p_2$.
\end{ex}

Note that, by \cite[Lemma 5.1]{duchamp+kalka}, the self-intersection $C\cdot C$ should always be negative under the assumption $i_C^*\eta\not\equiv 0$.
With this in mind, the theorem above seems incomplete, and for some reason, the form $\eta_\F$ cannot have an arbitrarily high number of zeroes for $C$ fixed.
It is not clear what values are admissible for $C\cdot C$ in $[1-g,-1]$, but to get an idea of this, it seems necessary to find examples which are not ramified coverings.

\begin{thm}
\label{thm_complete_leaves}
Suppose that $L$ is a leaf of $\F$ with compact adherence in $S$.
Suppose moreover that $i_C^*\eta_\F$ is not identically zero for any leaf $C$ in the adherence $\overline{L}$.
Then $L$ is complete for the metric $|\eta_\F|^2$.
\end{thm}

\begin{proof}
Suppose $i_L^*\eta_\F$ is not identically zero.
Then the set of points on $L$ for which the metric degenerates is discrete, and we want to prove that for any geodesic $\gamma: \mathbb{R}^+ \rightarrow L$ which tends to the border of $L$, the length $\ell(\gamma) := \int_0^\infty |\eta(\gamma'(t))|dt$ is infinite.

Since the adherence of $L$ is compact, we can cover it by a finite number of open sets $U_i$ with $\eta_i = b_i dx_i$ on $U_i$.
Remark first that for $U_i$ small enough, we can suppose that $\gamma$ passes through an infinite number of these $U_i$.
On each leaf, the set of points where $b_i=0$ is discrete, so we can find an $\varepsilon>0$ such that $B_i:=\{|b_i|<\varepsilon\}$ is relatively compact in $U_i$ on each leaf.
If $\gamma$ passes an infinite number of times through these $B_i$, then $\gamma$ must make an infinite number of times the path from within $B_i$ to the border of $U_i$, and since $|b_i|\geq \varepsilon$ on the complement $U_i\setminus B_i$, we must have $\ell(\gamma)=\infty$.

Thus we can suppose that $\gamma$ only meets a finite number of these $B_i$.
Then after a finite time $T>0$, the geodesic $\gamma$ only stays in the regions $|b_i|\geq \varepsilon$.
The result follows easily.
\end{proof}

\section{Semirank}
\label{sec_semirank}
\subsection{Definitions}

We will define the semirank using the construction of jet spaces; let us recall this construction in our setting.
Consider a holomorphic manifold $X$ equipped with a holomorphic field of holomorphic 2-planes $\mc{P}$.

We can consider the field of 2-planes $\mc{P}$ as a rank 2 subbundle of $TX$; let $X' = \mathbb{P}\mc{P}$ be the $\mathbb{P}^1$-bundle obtained by projectivizing the rank 2 linear bundle $\mc{P}$.
Then $X'$ comes with a fibration $p: X' \rightarrow X$, a field of 3-planes $p^*\mc{P}$ and a field of 2-planes $\mc{P}'$ inside $p^*\mc{P}$ such that the value of $\mc{P}'$ at a point $(x,\lambda)\in X'$ is given by $\mr{Ker}(dv - \lambda du)\subset p^*\mc{P}$ if $(u,v)$ are linear variables on $\mc{P}_x$ with $\lambda$ corresponding to the direction in the kernel of $dv - \lambda du$.

The manifold $X'$ satisfies the property that every holomorphic curve $C$ in $X$ tangent to $\mc{P}$ can be uniquely lifted to a curve $C'$ in $X'$ tangent to $\mc{P}'$ with $p(C') = C$.

Beginning by $(X_0,\mc{P}_0) = (U, TU)$, we can apply this construction inductively to obtain a sequence $(X_n,\mc{P}_n)$ of $(n+2)$-dimensional manifolds equipped with holomorphic fields of 2-planes and fibrations $p_n: X_{n} \rightarrow X_{n-1}$ with fibers $\mathbb{P}^1$.
Consider $\pi_n: X_n \rightarrow U$ the composition of the $p_n$.
By construction, each holomorphic curve in $U$ can be uniquely lifted to a curve in $X_n$ tangent to $\mc{P}_n$.

Now, suppose that $\F$ is a semiholomorphic foliation on $U$.
It can be considered as a family of holomorphic curves with two real parameters; as such it can be lifted to each $X_n$ as a family of curves tangent to $\mc{P}_n$ and defines on each $X_n$ a real 4-dimensional submanifold $Y_n$ of $X_n$.
Note that any point $p\in U$ around which $\F$ is smooth can be uniquely lifted to the point in $X_n$ corresponding to the lift of the leaf of $\F$ passing through $p$.
By abuse of notation, we will still write $p$ this point whenever the foliation is clear from the context.

\begin{df}
We define the \emph{semirank} of a semiholomorphic foliation $\F$ around a point $p\in U$ as the lowest integer $n$ such that the germ of $Y_n$ at $p$ is not Zariski-dense in $X_n$.
\end{df}

By convention, the semirank is infinite if there exists no such integer.
Note that holomorphic foliations correspond to semirank 1.

\begin{prop}
Suppose that $\F$ is of semirank $n<\infty$ around a point $p$.
Then the Zariski closure $Z$ of $Y_n$ in $X_n$ is a hypersurface generically transverse to the fibers of $p_n: X_n \rightarrow X_{n-1}$
\end{prop}

\begin{proof}
Suppose on the contrary that $Z$ is tangent to the fibration $p_n$ or is not a hypersurface.
Then $Z_{n-1} := p_n(Z)$ is a strict subvariety of $X_{n-1}$.
By construction, $Z_{n-1}$ contains $Y_{n-1}$, thus $Y_{n-1}$ is not Zariski-dense in $X_{n-1}$, contradicting the minimality of $n$.
\end{proof}

From this proposition, we see that $TZ\cap \mc{P}_n$ defines a foliation by curves in $Z$ tangent to $\mc{P}_n$ which can be projected to a holomorphic family of holomorphic curves in $U$ with $n$ parameters containing the leaves of $\F$.
By construction, this n-parameter holomorphic family of curves is uniquely determined by $\F$.

\begin{df}
If $\F$ is of semirank $n$, the n-parameter holomorphic family of curves given by the Zariski closure of $Y_n$ in $X_n$ is called the \emph{system of curves} $\ms{S}$ defined by $\F$.
We will also say that $\F$ is tangent to $\ms{S}$.
\end{df}

\begin{ex}
For the foliation $\mr{Im}(x)dy - \mr{Im}(y)dx$, the induced system of curves is the set of complex affine lines.
\end{ex}

To put this system into equations, take some coordinates $(x,y,\lambda_1,\ldots,\lambda_n)$ of $X_n$ centered around the point given by the leaf of $\F$ passing through the origin, so that $\lambda_k$ is the coordinate of the $\mathbb{P}^1$-fiber of $X_k$ corresponding to the direction $d \lambda_{k-1}/dx$.
Since the Zariski closure $Z$ of $Y_n$ is generically transverse to the fibration $p_n$, we can express it around a generic point as the graph of a holomorphic application
\[
\lambda_n = F(x,y,\lambda_1,\ldots,\lambda_{n-1}).
\]
In this case, the system of curves is given by solutions of the differential equation 
\[
y^{(n)}=F(x,y,y',\ldots,y^{(n-1)}),
\]
where a curve is written as a function $y(x)$ and derivations are made with respect to the variable $x$.
Note that some points can behave singularly for the system of curves, even if $\F$ is smooth.
This happens exactly when $Z$ is tangent to the fibration $p_n: X_n \rightarrow X_{n-1}$.
We say that the family $\ms{S}$ is \emph{smooth} when $Z$ is smooth and transverse to the fibration $p_n$.

\begin{ex}
The system of curves associated to the foliation $\mr{Re}(1+x^2)dy - 2x\mr{Re}(y)dx $ is the set of parabolas $\{y=ax^2+b\}$.
The differential equation satisfied by these curves is $xy''=y'$ so $Z$ has equation $x \lambda_2 - \lambda_1 = 0$ and is not transverse to the vector field $\frac{\partial_{}}{\partial_{}\lambda_2}$ above $x=0$.
\end{ex}

The proof of the following proposition is an adaptation of \cite[Chapitre I, \S I.4]{cartan_geometrie_pseudo_conforme_1} in the context of semiholomorphic foliations.

\begin{prop}
\label{prop_semirank_analytic}
Let $\F$ be a real-analytic semiholomorphic foliation.
Then the semirank of $\F$ is at most 2.
\end{prop}

\begin{proof}
Suppose that $(x,y)$ are coordinates around the origin such that $\F$ is nowhere vertical.
Consider coordinates $(x,y,\lambda_1,\lambda_2)$ of $X_2$ where $\lambda_i$ is a coordinate of the $\mathbb{P}^1$-fiber of $X_i$.
Let $\lambda(x,y)$ be the slope of $\F$.
The variety $Y_2$ satisfies the equations 
\[
\left\{ \begin{aligned} \lambda_1 &= \lambda(x,y) \\
\lambda_2 &= \frac{\partial_{}\lambda}{\partial_{}x}(x,y) + \lambda(x,y) \frac{\partial_{}\lambda}{\partial_{}y}(x,y). \end{aligned} \right.
\]

Consider the 4 real parameters $(x_1,x_2,v_1,v_2)$ with $l=l_1+il_2$ for each $l=x,v$, and $y(x,v)$ the solution of the equations
\[
\left\{ \begin{aligned} \frac{\partial_{}y}{\partial_{}x} &= \lambda(x,y)\\
\frac{\partial_{}y}{\partial_{}\bar{x}} &=0\\
y(0,v) &= v \end{aligned} \right.
\]
and put
\[
\varphi(x,v) = (x,y,\lambda(x,y), \partial_{x}\lambda(x,y)+\lambda(x,y)\partial_{y}\lambda(x,y)),
\]
where we wrote $y=y(x,v)$ for short.
Now remark that $\varphi$ is a parametrization of $Y_2$, and since $\lambda$ satisfies equation \eqref{eq_intgb_lambda}, the antiholomorphic derivative $\frac{\partial_{}\varphi}{\partial_{}\bar{x}}$ is zero.

Thus $\varphi$ is holomorphic in $x$, and real-analytic in $v_1,v_2$; as such it can be extended to a germ of holomorphic application $\varphi: (x,v_1,v_2) \to X_2$ defined in a neighborhood $V$ of $(\mathbb{C}\times \mathbb{R}\times \mathbb{R},0)$ in $(\mathbb{C}^3,0)$.
Then $\varphi(V)$ is a complex 3-dimensional subvariety of $X_2$ containing $Y_2$, which concludes the proof.
\end{proof}

Remark that when $\F$ is a semiholomorphic foliation and $L$ is a leaf of $\F$, the first order approximation of $\F$ along $L$ is always real analytic, thus holomorphic or of semirank 2.
If its system of curves is locally given by $\lambda' = F(x,y,\lambda)$, the first order approximation is the limit of $h_t^*\F$ when $t$ tends to $0$, where $h_t$ is the dilatation $h_t(x,y)=(x,ty)$.
The system of curves defined by the first order approximation is thus of the form $\lambda'= a_1(x)y + a_2(x)\lambda$; it is of order less than 3 in $\lambda$, so it is a projective structure.
However, in general this projective structure will have singular points along $L$.

After these generalities, we will restrict ourselves to foliations of semirank 2.

\subsection{Duality}
\label{sec_duality}


Suppose given a holomorphic family of curves $\ms{S}$ with $2$ parameters in an open set $U\subset \mathbb{C}^2$, represented by the hypersurface $Z\subset X_2$ and suppose that $Z$ is smooth and transverse to the fibration $p_2:X_2 \rightarrow X_{1}$.
We will consider $Z$ as a germ around the origin $0\in X_2$.
Write $\G$ the smooth holomorphic foliation of $Z$ given by $\ms{S}$, and $\check{U}$ the contraction of $Z$ in the direction $\G$.
The space $\check{U}$ is a holomorphic surface parametrizing the curves of $\ms{S}$; we call it the \emph{dual} of the system of curves $\ms{S}$.

This dual comes with a 2-parameter family of curves: for each point $z\in U$, consider the $\mathbb{P}^1$-fiber $F_z$ of $X_1$ over $z$.
Since $Z$ is transverse to the fibration $p_2$, the preimages of the $F_z$ on $Z$ form a smooth foliation by curves $\check{\mc{G}}$ on $Z$.
This foliation is transverse to $\mc{G}$, so that it is projected to a 2-parameter family of smooth curves on $\check{U}$.

This family is the dual family of $\ms{S}$ on $\check{U}$; since $U$ is the contraction of $Z$ in the direction $\check{\mc{G}}$, we see that $\mc{G}$ and $\check{\mc{G}}$ play symmetric roles on $Z$ so that the bidual of $\ms{S}$ is $\ms{S}$ itself.


To any semiholomorphic foliation $\F$ tangent $\ms{S}$, we can associate a real surface $S_\F\subset \check{U}$ by looking at its leaves as points in $\check{U}$ (or equivalently, by projecting $Y_2$ to $\check{U}$).
Of course, $\F$ is holomorphic if and only if $S_\F$ is a holomorphic curve.

Most of the time, this duality is incomplete, in the sense that both $U$ and $\check{U}$ are small open sets.
One can expect that when either $U$ or $\check{U}$ are globally well-defined, the situation is much more rigid.
For example, when the family $\ms{S}$ is a projective structure (meaning that through each point and each direction passes a curve of $\ms{S}$), then $\check{U}$ is a neighborhood of a $\mathbb{P}^1$ of self-intersection 1, and the automorphism group of the family $\ms{S}$ is finite-dimensional, see \cite{fl_projective_structures} for generic automorphisms, and the Zusatz of \cite[Satz 4]{commichau+grauert} to show that automorphisms are finitely determinated.

Let us consider for one moment the most particular case: when $\ms{S}$ is the family of affine lines.
We can thus suppose that $U= \mathbb{P}^2$ and $\check{U}=\check{\mathbb{P}}^2$, but the foliation is a priori only a foliation in the neighborhood of a point $0\in \mathbb{P}^2$.
Then $S_\F$ is a germ of a real surface in $\check{\mathbb{P}}^2$ and we can use \cite{thom_dualite} to form its dual $\check{S}_\F\subset \mathbb{P}^2$.
In general, $\check{S}_\F$ is the hypersurface given by the enveloppe of the family of curves $\F$.
The construction of \cite{thom_dualite} is in fact very general: we can do it for any system of curves $\ms{S}$.
In this case, the definition of the dual of a real subvariety $M\subset U$ is
\[
\check{M} = \{ L\in \check{U} \:|\: \text{$L$ intersects $M$ at a point $p$ with } T_pL+ T_p M\neq \mathbb{C}^2 \}.
\]

As stated above, the dual of a real surface is generically a real hypersurface.
More precisely, we can say the following:

\begin{lem}
Suppose $\ms{S}$ is a smooth 2-parameter system of curves in $U\subset \mathbb{C}^2$ and $S$ is a real surface in $U$.
We have the following possibilities for the dual $\check{S}$:
\begin{enumerate}
\item $\check{S}$ is a point, and $S$ is a complex curve in the family $\ms{S}$.
\item $\check{S}$ is a complex curve, and $S$ is also a complex curve.
\item $\check{S}$ is a real surface, in which case the intersection between curves of $\ms{S}$ and the surface $S$ define a real 2-parameter family of curves $\ms{S}_{\mathbb{R}}$ on $S$; when $S$ is real analytic, $\ms{S}$ is the complexification of $\ms{S}_{\mathbb{R}}$.
\item $\check{S}$ is a real hypersurface.
\end{enumerate}
\end{lem}

\begin{proof}
The projection $p_2:X_2 \rightarrow X_1$ induces a biholomorphism between the germs $(Z,0)$ and $(X_1,0)$ so we will work with $X_1$.
In particular, the families $\ms{S}$ and $\check{\ms{S}}$ give two smooth transverse foliations $\G$ and $\check{\G}$ on $X_1$, and the contact structure $\mc{P}_1$ is given by $\mc{P}_1 = T\G \oplus T\check{\G}$.
Consider also the two projections $p=p_1:X_1 \rightarrow U$ and $p':X_1 \rightarrow \check{U}$.

Now, as in \cite{thom_dualite}, we define the lift $p^*S\subset X_1$ of $S$ as the set of those $(z,\lambda)\in X_1$, where $z$ is a coordinate in $U$ and $\lambda$ a coordinate of the fiber, such that $z\in S$, and $\lambda$ is the complex direction of a real tangent vector in $T_zS$.
Note that $\check{S} = p'(p^*S)$.
As discussed in \cite[\S 4]{thom_dualite}, the lift $p^*S$ can be non-transverse to the projection $p'$ so that $\check{S}$ can be a priori very singular, in which case we could have difficulties defining the second lift $p'^*\check{S}$.
However, as soon as it is well-defined, we have the equality of germs $p^*S = p'^*\check{S}$.
In particular, if $\check{S}$ is a point, then $p^*S$ is a leaf of $\G$, and $S$ is a complex curve in the family $\ms{S}$.
We see also that when $\check{S}$ is a holomorphic curve, then $p^*S$ is the holomorphic legendrian lift of $\check{S}$, so that its projection $S=p(p^*S)$ is also a holomorphic curve.

When $\check{S}$ is a real surface which is not holomorphic, the lift $p^*S$ is a real 3-dimensional subvariety of $X_1$ which is everywhere non-transverse to both $p$ and $p'$.
In particular, the fibers of $p'$ define a foliation by real curves $\G_{\mathbb{R}}$ on $p^*S$.
This foliation is generically transverse to the projection $p$ so $p(\G_{\mathbb{R}})$ is a foliation by real curves on $S$.
By definition, the leaves of $\G_{\mathbb{R}}$ are included in leaves of $\G$ so that the leaves of the induced foliation on $S$ are included in curves of the system $\ms{S}$.
When $S$ is real analytic, we can suppose modulo biholomorphism that $S=\mathbb{R}^2$ around a generic point, and in this case it is clear that $\ms{S}$ is the complexification of $\ms{S}_{\mathbb{R}}$. 
\end{proof}

Note however that if $\F$ is a smooth semiholomorphic foliation in an open set $U$ and $S_\F$ is its dual surface, then the bidual $\check{S}_\F$ will not intersect $U$.
We will explain this in more details in the next section, but for now, let us study some cases when $\check{S}_\F$ does intersect $U$, which is very exceptional, as we see in the following result:

%
%

\begin{prop}
\label{prop_sing}
Let $\ms{S}$ be a smooth system of curves in an open set $U\subset \mathbb{C}^2$ and $\F$ a semiholomorphic foliation tangent to $\ms{S}$ with a singular point $p\in U$.
If $p$ is an isolated singular point, then $\F$ is the pencil of curves of $\ms{S}$ passing through $p$; it is holomorphic and $\check{S}_{\F}=\{p\}$.
Suppose $p\in U$ is a non-isolated singular point of $\F$.
Then the bidual $\check{S}_\F$ is a real surface passing through $p$, $\mr{Sing}(\F)=\check{S}_\F$, the system of curves $\ms{S}$ induces a real system of curves $\ms{S}_{\mathbb{R}}$ on $\check{S}_\F$ and when $\F$ is real analytic, $\ms{S}$ is the complexification of $\ms{S}_{\mathbb{R}}$.
\end{prop}

\begin{proof}
Since $\F$ is a foliation, the real codimension of its singular set $\Sigma$ is at least 2.
Consider a leaf $L_0$ of $\F$ intersecting $\Sigma$ at a point $p_0$.
By definition, $\F$ has another leaf $L_1$ intersecting $L_0$ at $p_0$.
Since $L_0$ and $L_1$ are curves in $\ms{S}$ and $\ms{S}$ is smooth, their interesction at $p_0$ is transversal, and any leaf $L$ of $\F$ close to $L_1$ intersects $L_0$ at a point $p\in \Sigma$ close to $p_0$.
If this point $p$ is always equal to $p_0$, then all of the leaves pass through $\{p_0\}$ and $\Sigma=\{p_0\}$.
In this case we see that $\check{\Sigma} = S_\F$ is the holomorphic curve of the family $\check{\ms{S}}$ corresponding to $p_0$ so that $\F$ is holomorphic.

Else the set of these points $p$ is a curve on $\Sigma$ and in particular $L\cap \Sigma$ is a curve.
We can say the same thing for all leaves, so that each leaf $L$ intersects $\Sigma$ along a curve.
Since the intersection points between leaves are points, $\Sigma$ cannot be a real curve and it must be a real surface.
Necessarily, these intersections define a real 2-parameter family of real curves on $\Sigma$, that is, a real system of curves $\ms{S}_{\mathbb{R}}$.
In particular, there is an open set $V$ in the real tangent bundle $T \Sigma$ such that for $(p,v)\in V$, there is a leaf $L$ of $\F$ intersecting $\Sigma$ at $p$ in the real direction $v$.
By definiton of the dual of $\Sigma$, we have $\check{\Sigma} = S_\F$, and thus $\Sigma = \check{S}_\F$.

Finally, note that the curves of the system $\ms{S}_{\mathbb{R}}$ are all contained in leaves of $\F$, so that the complexified of $\ms{S}_{\mathbb{R}}$ is a complex system of curves containing all the leaves of $\F$.
Since $\F$ is not holomorphic, it must be equal to $\ms{S}$.
\end{proof}

\begin{prop}
Let $C$ be an elliptic curve and $S$ a 2-dimensional neighborhood of $C$.
Suppose there exists a smooth semiholomorphic foliation $\F$ of semirank 2 in $S$ admitting $C$ as a leaf.
Suppose also that the system of curves $\ms{S}$ defined by $\F$ is smooth in the neighborhood of the curve $C$, and the dual system $\check{\ms{S}}$ is a projective structure.
Then $\F$ is holomorphic.
\end{prop}

\begin{proof}
Consider the universal cover $(U,\tilde{C})$ of $(S,C)$, equipped with the pullback semiholomorphic foliation, and the induced system of curves which we will still denote by $\F$ and $\ms{S}$.
If the dual system $\check{\ms{S}}$ is a projective structure, then the dual $V\supset U$ of $\check{U}$ contains a compact curve $L\simeq \mathbb{P}^1$; the curve $L$ is the compactification $L=\tilde{C}\cup\{\infty\}$, the surface $V$ is a neighborhood of $L$, we have $L\cdot L = 1$, and the system $\ms{S}$ extends naturally as a 2-parameter family of deformations of $L$.

It follows that any curve $L'\in \ms{S}$ close to $L$ intersects $L$ at one point, so that if the foliation $\F$ is smooth, then $L'\cap L \in L\setminus \tilde{C}=\{\infty\}$.
We see that $\F$ is exactly the pencil of curves $L'\in \ms{S}$ which intersect $L$ at infinity, so that it is indeed holomorphic.
\end{proof}

\subsection{Complete local models}

As we saw in Theorem \ref{thm_complete_leaves}, in many interesting cases, the leaves will be complete for the metric $|\eta_\F|^2$.
It would then be interesting to study local models of foliations $\F$ in open sets $U\subset \mathbb{C}^2$ such that the leaves of $\F$ are complete.

\begin{ex}
Consider the foliation given by the 1-form $\omega = \mr{Im}(x)dy - \mr{Im}(y)dx$.
This is in fact a singular foliation on the whole of $\mathbb{P}^2(\mathbb{C})$.
Its singular set is equal to $\mr{Sing}(\F) = \mathbb{P}^2(\mathbb{R})$.
Each complex line $L$ tangent to the foliation is cut in two pieces by $\mr{Sing}(\F)$, and each piece equipped with the metric $|\eta|^2$ is equal to Poincaré's half plane.
\end{ex}

This example is in fact very special: as we saw in Proposition \ref{prop_sing}, it corresponds to the case when the dual of $\check{S}_\F$ is a real surface, which is a degenerate case.
In the rest of this section, we suppose that we are in the generic case when $\check{S}_\F$ is a real hypersurface.

By the duality, the hypersurface $\check{S}_\F$ is the enveloppe of the family of curves parametrized by $S_\F$: each leaf $L$ of $\F$ intersect $\check{S}_\F$ along a curve $\gamma\subset L$, and is tangent to $\check{S}_\F$ along $\gamma$ (see \cite{thom_dualite} for more details); it follows that the leaves of $\F$ are tangent to $\check{S}_\F$ but we cannot extend them outside the curve $\gamma$ or different leaves will intersect $L$.
This behaviour should be the generic behaviour for complete models, even when the foliation is not of semirank 2.

More precisely, consider an open set $V\subset \mathbb{C}^2$, a smooth 2-parameter family of curves $\ms{S}$, and a germ of semiholomorphic foliation $\F$ along a leaf $d_0$, tangent to $\ms{S}$ and giving a germ of real surface $(S_\F,d_0)$ in the dual surface $\check{V}$.
Now suppose that the family $\ms{S}$ can be extended to a smooth system still denoted $\ms{S}$ on some open set $U\supset V$, and that the dual $\check{S}_\F$ is a real hypersurface in $U$ which intersects $d_0$ along a real curve.
Inside the germ of surface $(U,d_0)$, the hypersurface $\check{S}_\F$ has an interior which can be considered the biggest domain of definition of $\F$, and each leaf of $\F$ cuts tangencially $\check{S}_\F$ along a real curve.
Note that, using notations of section \ref{sec_motion}, this real curve has equation $\{|u|^2=|v|^2\}$ so that a generic geodesic along a leaf which tends to $\check{S}_\F$ has infinite length.
Thus, when the real curve $\check{S}_\F\cap d_0$ is compact, $\F$ can be extended to a semiholomorphic foliation in the interior of $\check{S}_\F$ whose leaves are complete.
Conversely, note that if leaves of $\F$ do not adhere to the hypersurface $\check{S}_\F$ then they are not complete.

\begin{figure}[H]
\includegraphics[scale=0.5]{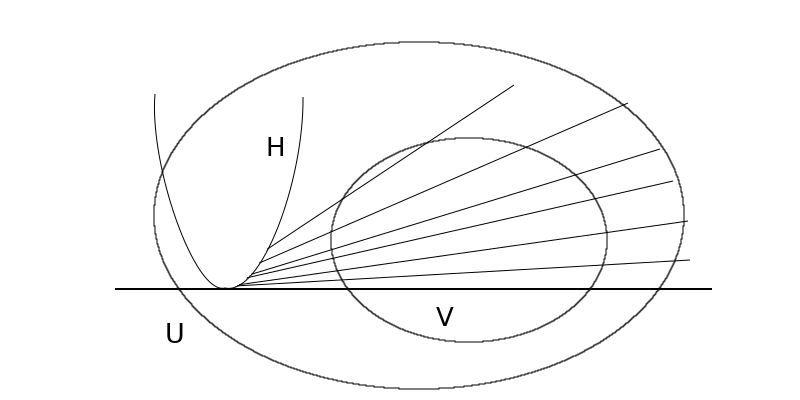}
\caption{The hypersurface $H=\check{S}_\F$ is tangent to leaves of $\F$.}
\end{figure}

\begin{df}
We call germ of \emph{smooth non-degenerated local complete model} along a leaf $L$ the triples $(U,V,\F)$ where $V\subset U\subset \mathbb{C}^2$ are germs of open sets along $L$, $\F$ is a germ of smooth semiholomorphic foliation on $V$ tangent to a system of curves $\ms{S}$, $L\cap V$ is a leaf of $\F$, the system $\ms{S}$ can be extended to a smooth system of curves on $U$, and the border $H=\partial_{}V$ is a germ of compact hypersurface which satisfies $H = \check{S}_\F$ as germs along $L\cap H$.
\end{df}

If a smooth semiholomorphic foliation tangent to a regular system of curves on a global surface has complete leaves, then its universal covering is either of the form described in Proposition \ref{prop_sing}, or is a smooth non-degenerated local complete model along each leaf.

\section{Foliations tangent to projective structures}
\label{sec_projective_structures}

\subsection{Examples of foliations by lines}
\label{sec_ex_lines}

We try to construct some examples of foliations $\F$ on germs of surfaces $(S,C)$ around a hyperbolic compact complex curve $C$, such that $C$ is a leaf of $\F$.
To do so, we will take $\F$ locally modelled on the foliation $\F_0$ given by the $(1,0)$-form $\omega_0 = \mr{Im}(x)dy - \mr{Im}(y)dx$ on $\mathbb{P}^2(\mathbb{C})$ around the leaf $L_0 = \{y=0\}$.
Equivalently, we want to find a group $G$ of germs of diffeomorphisms of the surface $\mathbb{P}^2(\mathbb{C})$ around an open set $V_0\subset L_0$, such that the quotient of $V_0$ by $G\vert_{V_0}$ is diffeomorphic to $C$.
Of course, if we want the foliation $\F_0$ to give a foliation $\F$ on the quotient, the group $G$ must send leaves of $\F_0$ to leaves of $\F_0$.
This means in particular that $G$ must preserve the set of affine lines in $\mathbb{P}^2(\mathbb{C})$, so $G$ is a subgroup of $\mr{PSL}_3(\mathbb{C})$.

The group $G$ must also preserve the singular set $\mathbb{P}^2(\mathbb{R})$ of $\F_0$, so $G\subset \mr{PSL}_3(\mathbb{R})$.
Note that by the duality explained in section \ref{sec_duality}, the fact for $G$ to preserve $\mathbb{P}^2(\mathbb{R})$ is equivalent for its action on the dual $\check{\mathbb{P}}^2(\mathbb{C})$ to preserve the dual surface $S_\F$.
This means exactly that $G$ sends leaves of $\F_0$ to leaves of $\F_0$, i.e. the group of automorphisms of $\F_0$ is $\mr{PSL}_3(\mathbb{R})$.
Write an element $M\in \mr{SL}_3(\mathbb{R})$ as
\[
M = \begin{bmatrix} a_{11} & a_{12} & a_{13}\\ a_{21} & a_{22} & a_{23}\\ a_{31} & a_{32} & a_{33} \end{bmatrix}.
\]
If we want $M$ to stabilize $L_0$, we must have $a_{21}=a_{23}=0$, and $M\vert_{L_0}$ is given in an affine coordinate $x$ by 
\[
M\vert_{L_0}(x) = \frac{a_{11}x+a_{13}}{a_{31}x+a_{33}}.
\]

Hence the construction: take any fuchsian subgroup $G_0\subset \mr{PSL}_2(\mathbb{R})$ such that the quotient of the half-plane $\mb{H}_0\subset L_0$ by $G_0$ is a hyperbolic compact curve $C$.
Write $a_{11}(g), a_{13}(g), a_{31}(g), a_{33}(g)$ the coefficients of the elements $g\in G_0$.
Choose any extension $G$ of $G_0$ to $\mr{PSL}_3(\mathbb{R})$ (i.e. $a_{22}(g)=1$, and $(a_{12}(g), a_{32}(g))$ is a cocycle for the group $G_0$); the most simple extension being of course $a_{12}(g)=a_{32}(g)=0$ and $a_{22}(g)=1$.
Then the quotient of a neighborhood $U$ of $\mb{H}_0$ in $\mathbb{P}^2(\mathbb{C})$ by $G$ is a surface $S$ containing a curve $C$ quotient of $\mb{H}_0$, and the foliation $\F_0$ descends to a smooth semiholomorphic foliation $\F$ on $S$ having $C$ as a leaf.

\begin{thm}
\label{thm_foliation_lines_examples}
Let $C = \mb{H}_0/G_0$ be a compact curve of genus $g\geq 2$, where $G_0\subset \mr{PSL}_2(\mathbb{R})$ is a Fuchsian subgroup.
Let $\ms{M}$ denote the moduli space of neighborhoods $(S,C,\F)$ of $C$ in complex surfaces equipped with a smooth semiholomorphic foliation $\F$ locally diffeomorphic to $\mr{Im}(x)dy-\mr{Im}(y)dx$, modulo biholomorphism.
The construction explained above induces a bijection
\[
\ms{M}\simeq H^1(G_0,\mathbb{R}^2).
\]
\end{thm}

\begin{proof}
As we explained above, every example comes from a cocycle $Z^1(G_0,\mathbb{R}^2)$, so that we only need to prove that equivalence modulo biholomorphism for $(S,C,\F)$ corresponds to equivalence modulo coboundaries for cocycles.
We keep the notations above in this proof.

The holonomy of $\F$ in $S$ defines a permutation of the leaves, which is exactly the action of $G$ on the dual surface $S_\F$; in particular, from this point of view the holonomy is given by elements of $\mr{PSL}_3(\mathbb{R})$.
It follows that this holonomy is well-defined modulo conjugacy in $\mr{PSL}_3(\mathbb{R})$; we can suppose that this conjugacy fixes $L_0$ and the group $G_0 = G\vert_{L_0}$.
The action by conjugacy of the subgroup of $\mr{PSL}_3(\mathbb{R})$ fixing $L_0$ and $G_0$ is exactly the action of coboundaries on cocycles, so that the theorem is proved.
\end{proof}

Note that, if we have coordinates $[x:y:z]$ on $\mathbb{P}^2$, in the affine chart $z=1$ the action of the group given by the zero cocycle writes 
\[
M(x,y) = \left( \frac{a_{11}x+a_{13}}{a_{31}x+a_{33}}, \frac{y}{a_{31}x+a_{33}} \right) =: (\alpha(x),\beta(x)y).
\]
The tangent bundle of the leaf $L_0=\{y=0\}$ will be given by $\alpha'(x) = \frac{a_{11}a_{33}-a_{13}a_{31}}{(a_{31}x+a_{33})^2}$ and its normal bundle by $\beta(x)$.
We see here explicitely that the tangent bundle and the normal bundle are closely related, as stated in Theorem \ref{thm_1-g}.

Note also that the cocycles only intervene at higher order.

\subsection{Generic foliations by lines}
\label{sec_generic_foliation_by_lines}

\begin{thm}
\label{thm_generic_foliation_lines}
Consider a real-analytic semiholomorphic foliation by lines $\F$ defined in an open set $U\subset \mathbb{P}^2(\mathbb{C})$ neighborhood of the leaf $L_0$.
Suppose there is a subgroup $G < \mr{PSL}_3(\mathbb{C})$ stabilizing $\F$ and $L_0$, such that the restriction to $L_0$ is injective, $G\vert_{L_0}$ is a Fuchsian group and the quotient of $L_0$ by $G$ is a compact Riemann surface of genus $g\geq 2$.

Suppose moreover that $\F$ is not holomorphic at first order along $L_0$, i.e. $i_{L_0}^*\eta_{\F}\not\equiv 0$.
Then $\F$ is biholomorphic to the foliation given by the $(1,0)$-form $\mr{Im}(x)dy - \mr{Im}(y)dx$.
\end{thm}

The foliation $\F$ defines a real surface $S=S_\F$ in $\check{\mathbb{P}}^2$, and the action of $G$ on the dual space $\check{\mathbb{P}}^2$, denoted by $\check{G}< \mr{PSL}_3(\mathbb{C})$ stabilizes $S$.
Since $G$ stabilizes a line $L_0$, the group $\check{G}$ stabilizes a point $p_0$.
If $g\in G$ is hyperbolic in restriction to $L_0$, then its action $\varphi\in \mr{PSL}_3(\mathbb{C})$ on $\check{\mathbb{P}}^2$ is such that $d_{p_0}\varphi$ is hyperbolic.
Most real surfaces do not have a lot of symmetries in $\mr{PSL}_3(\mathbb{C})$, we begin by examining those having one hyperbolic symmetry.

\begin{lem}
\label{lem_automorphisme_surface}
Let $(S,p_0)\subset \check{\mathbb{P}}^2$ be a smooth germ of real-analytic surface such that $T_{p_0}S$ is not complex.
Suppose that $\varphi\in \mr{PSL}_3(\mathbb{C})$ stabilizes $S$ and $d_{p_0}\varphi$ is hyperbolic.
Then either $S$ is a real affine plane, or $S$ is equipped with a real codimension 1 foliation whose leaves are invariant by $\varphi$.
Moreover, these leaves are intersections between $S$ and real affine planes passing through $p_0$; the foliation has exactly two separatrices at $p_0$, and they are tangent to eigenvectors of $d_{p_0}\varphi$.
\end{lem}

\begin{proof}
Since $\varphi(p_0)=p_0$, we can write $\varphi$ in some coordinates 
\[
\varphi = \begin{bmatrix} a_{11} & a_{12} & 0\\ a_{21} & a_{22} & 0\\ a_{31} & a_{32} & 1\end{bmatrix}\in \mr{GL}_3(\mathbb{C}).
\]
Now, the differential $d_{p_0} \varphi = \begin{bmatrix} a_{11} & a_{12}\\ a_{21} & a_{22}\end{bmatrix}$ is hyperbolic, and thus diagonalizable: $d_{p_0} \varphi \sim \begin{bmatrix} \lambda & 0 \\ 0 & \mu\end{bmatrix}$, for example with $|\lambda| > |\mu|$.
By hypothesis, $T_{p_0}S$ is not a complex direction, so there is a conjugacy $\psi$ with $\psi\in \mr{PSL}_3(\mathbb{C})$ fixing $p_0$, such that $T_{p_0}\psi(S) = p_{0}+\mathbb{R}^2$.
This implies $d_{p_0}(\psi\varphi\psi^{-1})(\mathbb{R}^2) = \mathbb{R}^2$, and since $\lambda$ and $\mu$ are not complex conjugates, we see that $\lambda$ and $\mu$ are real.

There are two cases to examine: either $\varphi$ is diagonalizable and in some coordinates,
\[
\varphi = \begin{bmatrix} \lambda & 0 & 0\\ 0 & \mu & 0\\ 0 & 0 & 1\end{bmatrix},
\]
or
\[
\varphi = \begin{bmatrix} \lambda & 0 & 0\\ 0 & 1 & 0\\ 0 & 1 & 1\end{bmatrix}.
\]
In both cases, if we consider coordinates $(x,y)$ centered at $p_0$ with $d_0 \varphi(x,y) = (\lambda x, \mu y)$, if we write $x=x_1+ix_2$, $y=y_1+iy_2$, then for any equations of the tangent plane $T_{p_0}S$
\[
\left\{ \begin{aligned}
l_{11}x_1 + l_{12}x_2 +l_{13}y_1 +l_{14}y_2 &= 0\\
l_{21}x_1 + l_{22}x_2 +l_{23}y_1 +l_{24}y_2 &=0,
\end{aligned} \right.
\]
by the equation $d \varphi(T_{p_0}S)=T_{p_0}S$, we get
\[
\left\{ \begin{aligned}
(\lambda/\mu -1) (l_{11}x_1 + l_{12}x_2) &=0\\
(\mu/\lambda -1) (l_{23}y_1 + l_{24}y_2) &=0.
\end{aligned} \right.
\]
We deduce that when $T_{p_0}S$ is not complex, it has equations $\{x_2/x_1 = \mr{tan}(\alpha_1), y_2/y_1 = \mr{tan}(\alpha_2)\}$ for some constants $\alpha_1,\alpha_2\in S^1$.
The application $(x,y,z) \mapsto (e^{-i \alpha_1}x,y,z)$ stabilizes $\varphi$ so we can suppose $\alpha_1=0$ in both cases.
In the diagonalizable case, we can also suppose $\alpha_2=0$.

Suppose now that we are in the diagonalizable case: in some coordinates $(x,y)$ centered at $p_0$, we have $\varphi(x,y) = (\lambda x, \mu y)$.
Remark first that the real functions $x_2/x_1$ and $y_2/y_1$ are stable by $\varphi$ and induce real foliations by invariant curves on $S$ whenever they are not constant on $S$.
We can write $S$ as a graph
\[
\left\{ \begin{aligned}
x_2 &= x_1f_1(x_1,y_1)\\
y_2 &= y_1f_2(x_1,y_1).
\end{aligned} \right.
\]
The functions $f_1,f_2$ are unique, and since $S$ is stable by $\varphi$, we get for every $x_1,y_1$,
\[
\left\{ \begin{aligned}
f_1(x_1,y_1) &= f_1(\lambda x_1, \mu y_1)\\
f_2(x_1,y_1) &= f_2(\lambda x_1, \mu y_1).
\end{aligned} \right.
\]
There are two cases: either the diffeomorphism $\varphi\vert_{\mathbb{R}^2}$ has a non-constant real-analytic first integral or not.
In the first case, every first integral is of the form $f(x_1^py_1^q)$ for some integers $p,q$, so that the two foliations induced by $x_2/x_1$ and $y_2/y_1$ on $S$ are in fact equal.
This means that every leaf of this common foliation is a component of the intersection between $S$ and a real affine plane $\{x_2=\mr{tan}(\theta)x_1, y_2 = \mr{tan}(\alpha)y_1\}$.
Note that, since this foliation has a first integral $x_1^py_1^q$, it has two separatrices at $p_0$, and the complexification of the tangents at $p_0$ of these separatrices are the eigenvectors of $\varphi$ corresponding to $\lambda$ and $\mu$.
In the second case, every first integral is constant, so that $S$ is included in a real affine plane $\{x_2=\mr{tan}(\theta)x_1, y_2=\mr{tan}(\alpha)y_1\}$.

In the non-diagonalizable case, $\varphi$ can be expressed in some coordinates $(x,y)$ centered at $p_0$ as
\[
\varphi(x,y) = \left( \frac{\lambda x}{1+y}, \frac{y}{1+y} \right).
\]
In the coordinates $(z,w) = (x/y, 1/y)$, this expression becomes $\varphi(z,w) = (\lambda z, w+1)$.
As before, we notice that the functions $z_2/z_1$ and $w_2$ are invariant by $\varphi$.

The tangent plane to $S$ can be parametrized by $(s,t)\in \mathbb{R}^2 \mapsto (x,y)=(s, te^{i \alpha_2})$, so that in a neighborhood of $p_0$, $S$ is close to the surface parametrized by $(z,w) = (s e^{-i \alpha_2}/t,e^{-i \alpha_2}/t)$.
In particular, if $\alpha_2\neq \pm\pi/2$, in a neighborhood of $p_0$ we can write $S\setminus \{p_0\}$ as a graph 
\[
\left\{ \begin{aligned}
z_2 &= z_1f_1(z_1,w_1)\\
w_2 &= f_2(z_1,w_1).
\end{aligned} \right.
\]
As in the diagonalizable case, the functions $f_i$ are unique and $S$ is stabilized by $\varphi$ so we get 
\[
\left\{ \begin{aligned}
f_1(z_1,w_1) &= f_1(\lambda z_1, w_1+1)\\
f_2(z_1,w_1) &= f_2(\lambda z_1, w_1+1).
\end{aligned} \right.
\]
The local first integrals of $\varphi\vert_{\mathbb{R}^2}$ are all of the form $f(z_1\mr{exp}(- w_1\mr{log}(\lambda)))$.
Note that $z_1\mr{exp}(- w_1\mr{log}(\lambda)) = (x/y) \mr{exp}(- \mr{log}(\lambda) /y)$ so that these first integrals are not real analytic at $(x,y)=(0,0)$.
It follows that $z_2/z_1$ and $w_2$ are constant on $S$.
Since $[x:y:1] = [z:1:w]$, these equations indeed define real affine subspaces.

If $\alpha_2$ is equal to $\pm \pi/2$, we can write $S$ as a graph 
\[
\left\{ \begin{aligned}
z_1 &= z_2f_1(z_2,w_2)\\
w_1 &= f_2(z_2,w_2),
\end{aligned} \right.
\] 
and we get
\[
\left\{ \begin{aligned}
f_1(z_2,w_2) &= f_1(\lambda z_2, w_2)\\
f_2(z_2,w_2) &= f_2(\lambda z_2,w_2)-1.
\end{aligned} \right.
\]
From the second equation we get that $f_2\vert_{\{z_2=0\}}=\infty$ so that in fact $\alpha_2\neq \pm \pi/2$.
\end{proof}

\begin{proof}[Proof of Theorem \ref{thm_generic_foliation_lines}]
Aiming at a contradiction, we suppose in this proof that the dual surface $S$ is not contained in any affine real plane.
The group $G\vert_{L_0}$ is generated by $2g$ hyperbolic elements $h_i$, and two different $h_i$ have different pairs of fixed points on the border $\partial_{}L_0$.
These fixed points are the eigenvectors of the differentials $d_{p_0}\varphi_i$ of the actions of $h_i$ on $\check{\mathbb{P}}^2$.

Consider 3 among them: $\varphi_1,\varphi_2,\varphi_3$.
The foliations defined by them on $S$ have different separatrices, so are different.
Thus by a generic point $p\in S$ pass three curves contained in real affine planes also passing through $p_0$.
Consider the projection $\pi : (\mathbb{P}^2(\mathbb{C}),p_0)\setminus \{p_0\} \rightarrow \mathbb{P}^3(\mathbb{R})$.
Since $S$ is not contained in any real plane, the space $\pi(S)$ is a surface; and we just saw that through a generic point $p\in \pi(S)$ pass three different affine lines contained in $\pi(S)$.
By \cite[\S 16.5]{ft_omnibus}, this implies that $\pi(S)$ is a real affine plane, so that $S$ is contained in a real affine hyperplane $H$.
This hyperplane $H$ is obviously unique, so that it is stable by $d_{p_0}\varphi$ for each $\varphi\in G$.
It follows that the unique complex direction tangent to $H$ is stable by each $d_{p_0}\varphi$, and that all of the elements of $G$ share an eigenvector.
This is obviously not possible for the fundamental group of a smooth compact curves.
\end{proof}

In the proof of this result, the hypothesis of real analyticity is only used in Lemma \ref{lem_automorphisme_surface} to show that when two real functions on $S$ are invariant by an automorphism $\varphi$, they must define the same foliation.
This result is false in the $\mc{C}^\infty$ context.
Indeed, consider a diffeomorphism $\varphi(x_1,y_1) = (\lambda x_1,\mu y_1)$ with for example $\lambda > 1 > \mu > 0$ and $\lambda^p\mu^q=1$ for some integers $p,q\in \mathbb{N}^*$.

Write $l=\mr{log}(\lambda), m=\mr{log}(\mu)$ and consider the functions $r = x^py^q$ and $\theta = \frac{1}{m-l}\mr{log} \left( \frac{y}{x} \right)$.
Then $r\circ \varphi = r$ and $\theta\circ \varphi = \theta+1$; the foliation $\{r=cte\}$ is stable by $\varphi$, has two separatrices $L_x = \{y=0\}, L_y=\{x=0\}$ and these separatrices cut the plane $\mathbb{R}^2$ in four invariant quadrants $\mathbb{R}^2\setminus (L_x\cup L_y) = \bigcup Q_i$.

If $f$ is a function defined in a quadrant $Q_i$ which is invariant by $\varphi$, its restriction to a curve $\{r=c\}$ is periodic in $\theta$, and we can develop it in Fourier series:
\[
f\vert_{Q_i}(r,\theta) = \sum_{n\in \mathbb{Z}} c_n^{(i)}(r)e^{in \theta}.
\]
We know by Fourier theory that $f\vert_{Q_i}$ is $\mc{C}^\infty$ if and only if $|c_n^{(i)}(r)| = o(n^k)$ for each $k\in \mathbb{N}$.
The question is then to find those tuples $(f_1,f_2,f_3,f_4)$ which can be glued to a $\mc{C}^\infty$ function $f$ on $\mathbb{R}^2$.
Remark here that if $|c_n^{(i)}(r)| = o(r^k)$ for each $i,n,k$, we can indeed glue the functions $f_i$ to a function $f$: when $(x,y)$ tends to one of the separatrices $L_x$ or $L_y$, a quick computation shows that all the derivatives $ \frac{\partial_{}^{k_1+k_2}f_i}{\partial_{}x^{k_1}\partial_{}y^{k_2}}$ tend to zero.

This shows that there are indeed germs of real surfaces invariant by an automorphism $\varphi$ which contradict Lemma \ref{lem_automorphisme_surface} in the $\mc{C}^\infty$ context.

\subsection{Another example}
\label{sec_another_example}

We can try to find other examples of semiholomorphic foliations whose leaves are curves of a projective structure.
In general, projective structures could have isolated symmetries, so we will only consider symmetries which are exponentials of infinitesimal automorphisms, as studied in \cite[\S 6.2]{fl_projective_structures}.
The only groups of infinitesimal automorphisms which can give rise to a compact curve of genus $g\geq 2$ are $\mf{sl}_2(\mathbb{C})$ and $\mf{sl}_3(\mathbb{C})$.
The only projective structure (modulo biholomorphisms) having a symmetry group $\mf{sl}_3(\mathbb{C})$ is the family of lines in $\mathbb{P}^2$, which we already considered.
As for $\mf{sl}_2(\mathbb{C})$, it only occurs as infinitesimal symmetries for the structure $\ms{S}$ whose curves have equations $y=y(x)$ in a neighborhood $U$ of the origin with 
\[
y'' = (xy' - y)^3.
\]
However the action of the group is the linear action of $G=\mr{SL}_2(\mathbb{C})$ so it fixes a point.
Thus we cannot build an example of neighborhood of a compact curve of genus $g\geq 2$ from this structure.

Consider then the dual family $\check{\ms{S}}$ of curves of this projective structure: it is defined in the surface $S_0$ described in \cite[\S 5.4]{fl_projective_structures}.
Each point $p\in U$ gives rise to a rational curve $L_p\subset S_0$, these curves all have self-intersection 1 in $S_0$, and they form the dual family.
We can consider $S_0$ as a neighborhood of $L_0$.
Since this family is the dual of $\ms{S}$ , it also has symmetry group $\check{G}=\mr{SL}_2(\mathbb{C})$; in $U$ the automorphism group fixed the origin $0\in U$, so that the action of $\check{G}$ on $S_0$ stabilizes the curve $L_0$, and acts as Möbius transforms on it.

Recall from \cite{fl_projective_structures} that $S_0$ is a double cover of $\mathbb{P}^1\times \mathbb{P}^1$ along the diagonal $\Delta$, so that the family $\check{\ms{S}}$ corresponds to graphs of Möbius functions tangent to $\Delta$.
In some coordinates $(x,y)$, a curve $L$ in $\check{\ms{S}}$ has equation 
\[
y = \frac{a (x-x_0)}{\sqrt{1+a^2(x-x_0)}},
\]
where the two parametres $x_0$ and $a$ correspond respectively to the point of intersection between $L$ and $L_0$, and to the slope of $L$ at this point.

Now consider the real plane $\{(x_0,a)\in \mathbb{R}^2\}\subset U$, which corresponds to a semiholomorphic foliation $\F$ in $S_0$ whose leaves are precisely those curves for which $a$ and $x_0$ are real.
We can check that when $a$ and $x_0$ are real, we can recover $x_0$ from the equation 
\[
x_0 = \frac{\mr{Im}(\bar{x}y^2) - |y|^2\mr{Im}(x)\sqrt{1-\frac{\mr{Im}(y^2)}{\mr{Im}(x)}}}{\mr{Im}(y^2)} =: f(x,y).
\]
In particular, the real function $f$ is constant along the leaves of $\F$ so that its level hypersurfaces are Levi-flat hypersurfaces foliated by leaves of $\F$.
We deduce that the $(1,0)$-form $\partial_{}f$ defines the foliation $\F$, and after some computations that the expression of the slope $\lambda = dy/dx$ is
\[
\lambda = \frac{\mr{Im}(y^2)}{2\mr{Im}(x)}\frac{\bar{y}\sqrt{1-\frac{\mr{Im}(y^2)}{\mr{Im}(x)}} + y \left(1 - \frac{\mr{Im}(y^2)}{2\mr{Im}(x)} \right)}{|y|^2\sqrt{1 - \frac{\mr{Im}(y^2)}{\mr{Im}(x)}} + \left( \mr{Re}(y^2) - \bar{y}^2 \frac{\mr{Im}(y^2)}{2\mr{Im}(x)} \right)}.
\]

The surface $S_0$ is defined over $\mathbb{R}$ and is the complexified of the real-analytic surface $S_{\mathbb{R}} = \{(x,y)\in \mathbb{R}^2\}$.
As before, $L_0\setminus S_{\mathbb{R}}$ is the union of two hyperbolic planes; consider one of them $\mb{H}_0 \subset L_0$ and a small neighborhood $V\subset S_0$ of $\mb{H}_0$.
One can check that if $V$ is small enough, $\F$ is smooth in $V$, and that its automorphism group is $\mr{SL}_2(\mathbb{R})$ acting as Möbius transformations on $\mb{H}_0$.

We conclude that for any compact curve $C$ of genus $g\geq 2$ there is exactly one example (up to biholomorphism) of a surface $(S,C)$ equipped with a smooth semiholomorphic foliation $\F$ coming from the construction above.
Remember that the construction above only gives those examples which come in family: there might be exceptional examples coming from automorphisms which are not exponentials of infinitesimal automorphisms.

\bibliography{semibib}{}
\bibliographystyle{acm}

\textsc{IMPA, Estrada Dona Castorina, 110, Horto, Rio de Janeiro, Brasil}

\textit{Email :} olivier.thom@impa.br

\end{document}